\documentclass[11pt]{amsart}
\usepackage{amsmath}
\usepackage[active]{srcltx}
\usepackage{t1enc}
\usepackage[latin2]{inputenc}
\usepackage{verbatim}
\usepackage{amsmath,amsfonts,amssymb,amsthm}
\usepackage[mathcal]{eucal}
\usepackage{enumerate}
\usepackage[centertags]{amsmath}
\usepackage{graphics}

\newtheorem{theorem}{Theorem}

\newtheorem{definition}{Definition}
\newtheorem{corollary}{Corollary}

\begin{document}

\author{Amiran Gogatishvili, Ushangi Goginava and George Tephnadze}
\title[Classes Of Functions]{Relations Between Some $\,$Classes Of Functions
Of Generalized Bounded Variation}
\address{A. Gogatishvili, Institute of Mathematics of the Academy of
Sciences of the Czech Republic, Zinta 25, 11567 Praha 1, Czech Republic}
\email{gogatish@math.cas.cz}
\address{U. Goginava, Department of Mathematics, Faculty of Exact and
Natural Sciences, Iv. Javakhishvili Tbilisi State University, Chavcha\-vadze
str. 1, Tbilisi 0128, Georgia}
\email{zazagoginava@gmail.com}
\address{G. Tephnadze, Departmentof Mathematics, Faculty of Exact and
Natural Sciences, Iv. Javakhishvili Tbilisi State University, Chavcha\-vadze
str. 1, Tbilisi 0128, Georgia}
\email{giorgitephnadze@gmail.com}
\date{}
\title{Relations Between Some $\,$Classes Of Functions
Of Generalized Bounded Variation}

\thanks{The reseach was supported by Shota Rustaveli National Science Foundation
grant no.13/06 (Geometry of function spaces, interpolation and embedding
theorems)\\The research of A. Gogatishvili was partly supported by the grants 201/08/0383 and 13-14743S
of the Grant agency of the Czech Republic. }

\begin{abstract}
We prove inclusion relations between generalizing Waterman's and generalized
Wiener's classes  for functions of two variable.
\end{abstract}

\maketitle

The notion of function of bounded variation was introduced by C. Jordan \cite%
{12}. Generalized this notion N. Wiener \cite{24} has considered the class $%
BV_{p}$ of functions. L. Young \cite{25} introduced the notion of functions
of $\Phi $-variation. In \cite{20} D. Waterman has introduced the following
concept of generalized bounded variation

\begin{definition}
Let $\Lambda =\left\{ \lambda _{n}:n\geq 1\right\} \,\ $be an increasing
sequence of positive numbers such that $\sum\limits_{n=1}^{\infty }\left(
1/\lambda _{n}\right) =\infty .$ A function $f$ is said to be of $\Lambda $%
-bounded variation $\left( f\in \Lambda BV\right) $, if for every choice of
nonoverlapping intervals $\{I_{n}:n\geq 1\}$, we have
\begin{equation*}
\sum\limits_{n=1}^{\infty }\frac{\left\vert f\left( I_{n}\right) \right\vert
}{\lambda _{n}}<\infty ,
\end{equation*}%
where $I_{n}=\left[ a_{n},b_{n}\right] \subset \left[ 0,1\right] $ and $%
f\left( I_{n}\right) =f\left( b_{n}\right) -f\left( a_{n}\right) .$
\end{definition}

If $\,\,f\in \Lambda BV,$ then $\Lambda $-variation of $f$ is defined to be
the supremum of such sums, denoted by $V_{\Lambda }\left( f\right) .$

Properties of functions of the class $\Lambda BV$ as well as the convergence
and summability properties of their Fourier series have been investigated in
\cite{17}-\cite{23}.

For everywhere bounded 1-periodic functions, Z. Chanturia \cite{6} has
introduced the concept of the modulus of variation.

H. Kita and K. Yoneda \cite{13} studied generalized Wiener classes $BV\left(
p\left( n\right) \uparrow p\right) .$ They introduced

\begin{definition}
Let $f$ be a finite $1$-periodic function defined on the interval $\left(
-\infty ,+\infty \right) $. $\Delta $ is said to be a partition with period $%
1$ if there is a set of points $t_{i}$ for which
\begin{equation}
\cdots t_{-1}<t_{0}<t_{1}<t_{2}<\cdots <t_{m}<t_{m+1}<\cdots ,  \label{1}
\end{equation}%
and $t_{k+m}=t_{k}+1$ when $k=0,\pm 1,\pm 2,...$, where $m$ is any natural
number. Let $p\left( n\right) $ be an increasing sequence such that $1\leq
p\left( n\right) \uparrow p$, $n\rightarrow \infty $, where $1\leq p\leq
+\infty $. We say that a function $f$ belongs to the class $BV\left( p\left(
n\right) \uparrow p\right) $ if
\begin{equation*}
V\left( f,p\left( n\right) \uparrow p\right) \equiv \sup\limits_{n\geq
1}\sup\limits_{\Delta }\left\{ \left( \sum\limits_{k=1}^{m}\left\vert
f\left( I_{k}\right) \right\vert ^{p\left( n\right) }\right) ^{1/p\left(
n\right) }:\inf\limits_{k}\left\vert I_{k}\right\vert \geq \frac{1}{2^{n}}%
\right\} <+\infty .
\end{equation*}
\end{definition}

We note that if $p\left( n\right) =p$ for each natural number, where $1\leq
p<+\infty $, then the class $BV\left( p\left( n\right) \uparrow p\right) $
coincides with the Wiener class $V_{p}$.

Properties of functions of the class $BV\left( p\left( n\right) \uparrow
p\right) $ as well as the uniform convergence and divergence at point of
their Fourier series with respect to trigonometric and Walsh system have
been investigated in \cite{8},\cite{10},\cite{14}.

Generalizing the class $BV\left( p\left( n\right) \uparrow p\right) $ T.
Akhobadze (see \cite{1},\cite{2}) has considered the $BV\left( p\left(
n\right) \uparrow p,\varphi \right) $ and$\,$ $B\Lambda \left( p\left(
n\right) \uparrow p,\varphi \right) $ classes of functions.

The relation between diferent classes of generalized bounded variation was
taken into account in the works of M. Avdispahic \cite{4}, A. Kovocik \cite%
{15}, A. Belov (see \cite{5}, Z. Chanturia \cite{7}), T. Akhobadze \cite{3}
and M. Medvedieva \cite{16}, Goginava \cite{9,AMH}.

Let $f$ be a real and measurable function of two variable of period $1$ with
respect to each variable. Given intervals $J_{1}=(a,b)$, $J_{2}=(c,d)$ and
points $x,y$ from $I:=[0,1]$ we denote
\begin{equation*}
f(J_{1},y):=f(b,y)-f(a,y),\qquad f(x,J_{2}):=f(x,d)-f(x,c)
\end{equation*}%
and for the rectangle $A=(a,b)\times (c,d)$, we set
\begin{equation*}
f(A)=f(J_{1},J_{2}):=f(a,c)-f(a,d)-f(b,c)+f(b,d).
\end{equation*}%
Let $E=\{I_{i}\}$ be a collection of nonoverlapping intervals from $I$
ordered in arbitrary way and let $\Omega $ be the set of all such
collections $E$.

For the sequence of positive numbers $\Lambda =\{\lambda
_{n}\}_{n=1}^{\infty }$ we denote
\begin{equation*}
\Lambda V_{1}(f)=\sup_{y\in I}\sup_{\{I_{i}\}\in \Omega }\sum_{i}\frac{%
|f(I_{i},y)|}{\lambda _{i}},
\end{equation*}%
\begin{equation*}
\Lambda V_{2}(f)=\sup_{x\in I}\sup_{\{J_{j}\}\in \Omega }\sum_{j}\frac{%
|f(x,J_{j})|}{\lambda_{j}}
\end{equation*}%
\begin{equation*}
\Lambda V_{1,2}(f)=\sup_{\{I_{i}\},\{J_{j}\}\in \Omega }\sum_{i}\sum_{j}%
\frac{|f(I_{i},J_{j})|}{\lambda _{i}\lambda _{j}}.
\end{equation*}

\begin{definition}
We say that the function $f$ has bounded $\Lambda $-variation on $%
I^{2}:=[0,1]\times \left[ 0,1\right] $ and write $f\in \Lambda BV$, if
\begin{equation*}
\Lambda V(f):=\Lambda V_{1}(f)+\Lambda V_{2}(f)+\Lambda V_{1,2}(f)<\infty .
\end{equation*}%
We say that the function $f$ has bounded Partial $\Lambda $-variation and
write $f\in P\Lambda BV$ if
\begin{equation*}
P\Lambda V(f):=\Lambda V_{1}(f)+\Lambda V_{2}(f)<\infty .
\end{equation*}
\end{definition}

If $\lambda _{n}\equiv 1$ (or if $0<c<\lambda _{n}<C<\infty ,\ n=1,2,\ldots $%
) the classes $\Lambda BV$ and $P\Lambda BV$ coincide with the Hardy class $%
BV$ and $PBV$ respectively. Hence it is reasonable to assume that $\lambda
_{n}\rightarrow \infty $ and since the intervals in $E=\{I_{i}\}$ are
ordered arbitrarily, we will suppose, without loss of generality, that the
sequence $\{\lambda _{n}\}$ is increasing. Thus, in what follows we suppose
that
\begin{equation}
1<\lambda _{1}\leq \lambda _{2}\leq \ldots ,\qquad \lim_{n\rightarrow
\infty, }\lambda _{n}=\infty, \qquad \sum_{n=1}^\infty
\frac1{\lambda_n}=\infty.  \label{Lambda}
\end{equation}

In the case when $\lambda _{n}=n,\ n=1,2\ldots $ we say \textit{Harmonic
Variation} instead of $\Lambda $-variation and write $H$ instead of $\Lambda
$ ($HBV$, $PHBV$, $HV(f)$, etc).

The notion of $\Lambda $-variation was introduced by Waterman \cite{20} in
one dimensional case and Sahakian \cite{Saha} in two dimensional case. The
notion of bounded partial variation (class $PBV$) was introduced by Goginava
\cite{GogEJA}. These classes of functions of generalized bounded variation
play an important role in the theory of Fourier series.

We have proved in \cite{GogSah} the following theorem.

\begin{theorem}[Goginava, Sahakian]
\label{T1} Let $\Lambda =\{\lambda _{n}=n\gamma _{n}\}$ and $\gamma _{n}\geq
\gamma _{n+1}>0,\ n=1,2,....\,\,\,$. \newline
1) If
\begin{equation}
\sum_{n=1}^{\infty }\frac{\gamma _{n}}{n}<\infty ,  \label{T1-1}
\end{equation}%
then $P\Lambda BV\subset HBV$.\newline
2) If, in addition, for some $\delta >0$
\begin{equation}
\gamma _{n}=O(\gamma _{n^{[1+\delta ]}})\quad \text{as}\quad n\rightarrow
\infty  \label{T1-10}
\end{equation}%
and
\begin{equation}
\sum_{n=1}^{\infty }\frac{\gamma _{n}}{n}=\infty ,  \label{T1-11}
\end{equation}%
then $P\Lambda BV\not\subset HBV$.
\end{theorem}

Dyachenko and Waterman \cite{DW} introduced another class of functions of
generalized bounded variation. Denoting by $\Gamma $ the the set of finite
collections of nonoverlapping rectangles $A_{k}:=\left[ \alpha _{k},\beta
_{k}\right] \times \left[ \gamma _{k},\delta _{k}\right] \subset T^{2}$ we
define
\begin{equation*}
\Lambda ^{\ast }V\left( f\right) :=\sup_{\{A_{k}\}\in \Gamma }\sum\limits_{k}%
\frac{\left\vert f\left( A_{k}\right) \right\vert }{\lambda _{k}}.
\end{equation*}

\begin{definition}[Dyachenko, Waterman]
Let $f$ be a real function on $I^{2}$. We say that $f\in \Lambda ^{\ast }BV$
if
\begin{equation*}
\Lambda V(f):=\Lambda V_{1}(f)+\Lambda V_{2}(f)+\Lambda ^{\ast }V\left(
f\right) <\infty .
\end{equation*}
\end{definition}

In \cite{GMJ} Goginava and Sahakian introduced a new class of functions of
generalized bounded variation and investigated the convergence of Fourier
series of function of that classes.

For the sequence $\Lambda =\{\lambda _{n}\}_{n=1}^{\infty }$ we denote
\begin{equation*}
\Lambda ^{\#}V_{1}(f)=\sup_{\{y_{i}\}\subset T}\sup_{\{I_{i}\}\in \Omega
}\sum_{i}\frac{|f(I_{i},{y_{i}})|}{\lambda _{i}},
\end{equation*}%
\begin{equation*}
\Lambda ^{\#}V_{2}(f)=\sup_{\{x_{j}\}\subset T}\sup_{\{J_{j}\}\in \Omega
}\sum_{j}\frac{|f(x_{j},J_{j}|}{\lambda _{j}}.
\end{equation*}

\begin{definition}[Goginava, Sahakian]
We say that the function $f$ belongs to the class $\Lambda ^{\#}BV$, if
\begin{equation*}
\Lambda ^{\#}V(f):=\Lambda ^{\#}V_{1}(f)+\Lambda ^{\#}V_{2}(f)<\infty .
\end{equation*}
\end{definition}

The following theorem was proved in \cite{GMJ}

\begin{theorem}
\label{T2} a) If
\begin{equation}
\overline{\lim\limits_{n\rightarrow \infty }}\frac{\lambda _{n}\log \left(
n+1\right) }{n}<\infty ,  \label{cond1}
\end{equation}%
then%
\begin{equation*}
\Lambda ^{\#}BV\subset HBV.
\end{equation*}%
b) If $\frac{\lambda _{n}}{n}\downarrow 0$ and%
\begin{equation*}
\overline{\lim\limits_{n\rightarrow \infty }}\frac{\lambda _{n}\log \left(
n+1\right) }{n}=+\infty ,
\end{equation*}%
then%
\begin{equation*}
\Lambda ^{\#}BV\not\subset HBV.
\end{equation*}
\end{theorem}

In this paper we introduce new classes of bounded generalized variation.

Let $f$ be a function defined on $R^{2}$ with $1$- periodic relative to each
variable. $\Delta _{1}$ and $\Delta _{2}$ is said to be a partitions with
period $1$, if
\begin{equation*}
\Delta _{i}:\cdots <t_{-1}^{\left( i\right) }<t_{0}^{\left( i\right)
}<t_{1}^{\left( i\right) }<\cdots <t_{m_{i}}^{\left( i\right)
}<t_{m_{i}+1}^{\left( i\right) }<\cdots ,\text{ \ \ }i=1,2
\end{equation*}%
satisfies $t_{k+m_{i}}^{\left( i\right) }=t_{k}^{\left( i\right) }+1$ for $%
k=0,\pm 1,\pm 2,...$, where $m_{i},$\ $i=1,2$ are a positive
integers.\medskip\ \newline

\begin{definition}
Let $p\left( n\right) $ be an increasing sequence such that $1\leq p\left(
n\right) \uparrow p$, $n\rightarrow \infty $, where $1\leq p\leq +\infty $.
We say that a function $f$ belongs to the class $BV^{\#}\left( p\left(
n\right) \uparrow p\right) $ if
\begin{align*}
&V_{1}^{\#}\left( f,p\left( n\right) \uparrow p\right)\\
&\hskip+0.8cm := \sup_{\{y_{i}\}\subset I}\sup\limits_{n\geq 1}\sup\limits_{\Delta
_{1}}\left\{ \left( \sum\limits_{i=1}^{m_{1}}\left\vert f\left(
I_{i},y_{i}\right) \right\vert ^{p\left( n\right) }\right) ^{1/p\left(
n\right) }:\inf\limits_{i}\left\vert I_{i}\right\vert \geq \frac{1}{2^{n}}%
\right\} <+\infty,\\
\intertext{and}
&V_{2}^{\#}\left( f,p\left( n\right) \uparrow p\right)\\
&\hskip+0.8cm := \sup_{\{x_{j}\}\subset I}\sup\limits_{n\geq 1}\sup\limits_{\Delta
_{2}}\left\{ \left( \sum\limits_{j=1}^{m_{2}}\left\vert f\left(
x_{j},J_{j}\right) \right\vert ^{p\left( n\right) }\right) ^{1/p\left(
n\right) }:\inf\limits_{j}\left\vert J_{j}\right\vert \geq \frac{1}{2^{n}}%
\right\} <+\infty ,
\end{align*}%
where%
\begin{equation*}
I_{i}:=\left( t_{i-1}^{\left( 1\right) },t_{i}^{\left( 1\right) }\right)
,J_{j}:=\left( t_{j-1}^{\left( 2\right) },t_{j}^{\left( 2\right) }\right) .
\end{equation*}
\end{definition}

$C\left( I^{2}\right) $ and $B\left( I^{2}\right) $ are the spaces of
continuous and bounded functions given on $I^{2}$, respectively.

In this paper we prove inclusion relations between $\Lambda ^{\#}BV$ and $%
BV^{\#}\left( p\left( n\right) \uparrow \infty \right) $ classes. In
particular, the following are true

\begin{theorem}
\label{th1}$\Lambda ^{\#}BV\subset BV^{\#}\left( p\left( n\right) \uparrow
\infty \right) $ if and only if
\begin{equation}
\overline{\lim\limits_{n\rightarrow \infty }}\sup\limits_{1\leq m\leq 2^{n}}%
\frac{m^{1/p\left( n\right) }}{\sum\limits_{j=1}^{m}\left( 1/\lambda
_{j}\right) }<\infty .  \label{2}
\end{equation}
\end{theorem}

\begin{theorem}
\label{th2}Let $\sum\limits_{n=1}^{\infty }\left( 1/\lambda _{n}\right)
=+\infty $ . Then there exists a functions \\$f\in BV^{\#}\left( p\left(
n\right) \uparrow \infty \right) \bigcap C\left( I^{2}\right) $ such that $%
f\notin \Lambda BV^{\#}.$
\end{theorem}

\begin{corollary}
\label{cor}$BV^{\#}\left( p\left( n\right) \uparrow \infty \right) \subset
\Lambda ^{\#}BV$ if and only if $\Lambda ^{\#}BV=B\left( I^{2}\right) .$
\end{corollary}

\begin{proof}[Proof of Theorem \protect\ref{th1}]
Let us take an arbitrary $f\in \Lambda ^{\#}BV$. Follow of the method of the paper
Kuprikov in \cite{kupr}, we can prove that the following estimations hold:
\begin{equation*}
\left( \sum\limits_{k=1}^{m_{1}}\left\vert f\left( I_{k},y_{k}\right)
\right\vert ^{p\left( n\right) }\right) ^{1/p\left( n\right) }\leq \Lambda
^{\#}V_{1}\left( f\right) \sup\limits_{1\leq m\leq 2^{n}}\frac{m^{1/p\left(
n\right) }}{\sum\limits_{i=1}^{m}\left( 1/\lambda _{i}\right) }<\infty
\end{equation*}

and
\begin{equation*}
\left( \sum\limits_{k=1}^{m_2}\left\vert f\left( x_{k},J_{k}\right)
\right\vert ^{p\left( n\right) }\right) ^{1/p\left( n\right) }\leq \Lambda
^{\#}V_{2}\left( f\right) \sup\limits_{1\leq m\leq 2^{n}}\frac{m^{1/p\left(
n\right) }}{\sum\limits_{i=1}^{m}\left( 1/\lambda _{i}\right) }<\infty .
\end{equation*}

Therefore, $f\in \Lambda ^{\#}BV\left( p\left( n\right) \uparrow \infty
\right) .$

Next, we suppose that the condition (\ref{2}) does not hold. As an example
we construct function from $\Lambda ^{\#}BV$ which is not in $BV^{\#}\left(
p\left( n\right) \uparrow \infty \right) .$

Since
\begin{equation*}
\overline{\lim\limits_{n\rightarrow \infty }}\sup\limits_{1\leq m\leq 2^{n}}%
\frac{m^{1/p\left( n\right) }}{\sum\limits_{j=1}^{m}\left( 1/\lambda
_{j}\right) }=+\infty ,
\end{equation*}
there exists a sequence of integers $\{n_{k}^{\prime }:k\geq 1\}$ such that
\begin{equation}
\lim\limits_{k\rightarrow \infty }\frac{m\left( n_{k}^{\prime }\right)
^{1/p\left( n_{k}^{\prime }\right) }}{\sum\limits_{j=1}^{m\left(
n_{k}^{\prime }\right) }\left( 1/\lambda _{j}\right) }=+\infty ,  \label{3}
\end{equation}
where
\begin{equation*}
\sup\limits_{1\leq m\leq 2^{n}}\frac{m^{1/p\left( n\right) }}{%
\sum\limits_{j=1}^{m}\left( 1/\lambda _{j}\right) }=\frac{m\left( n\right)
^{1/p\left( n\right) }}{\sum\limits_{j=1}^{m\left( n\right) }\left(
1/\lambda _{j}\right) }.
\end{equation*}

We choose a monotone increasing sequence of positive integers $\{n_{k}:k\geq
1\}\subset \{n_{k}^{\prime }:k\geq 1\}$ such that
\begin{equation}
\frac{m\left( n_{k}\right) ^{1/p\left( n_{k}\right) }}{\sum\limits_{j=1}^{m%
\left( n_{k}\right) }\left( 1/\lambda _{j}\right) }\geq 4^{k},  \label{4}
\end{equation}
\begin{equation}
p\left( n_{k}\right) \geq n_{k-1},  \label{5}
\end{equation}
\begin{equation}
n_{k}>3n_{k-1}+1\,\,\,\,\,\,\,\mbox{for all }k\geq 2.  \label{6}
\end{equation}

From (\ref{3}) and (\ref{5}) it is evident that $2^{2n_{k-1}}<m\left(
n_{k}\right) \leq 2^{n_{k}}.$

Two cases are possible:

a) Let there exists a monotone sequence of positive integers $\{s_{k}:k\geq
1\}\subset \{n_{k}:k\geq 1\}$ such that
\begin{equation}
2^{2s_{k-1}}<m\left( s_{k}\right) \leq 2^{s_{k}-s_{k-1}-1}.  \label{7}
\end{equation}

Consider the function $f_{k}$ defined by
\begin{equation*}
f_{k}\left( x\right) =\left\{
\begin{array}{l}
h_{k}\left( 2^{s_{k}}x-2j+1\right) ,x\in \left[ \left( 2j-1\right)
/2^{s_{k}},2j/2^{s_{k}}\right) \\
-h_{k}\left( 2^{s_{k}}x-2j-1\right) ,x\in \left[ 2j/2^{s_{k}},\left(
2j+1\right) /2^{s_{k}}\right) \\
\,\,\,\,\,\,\,\,\,\,\,\,\,\,\,\,\,\,\,\,\,\,\,\,\,\,\,\,\,\,\,\,\,\,\,\,\,\,%
\,\,\mbox{for\thinspace }j=m\left( s_{k-1}\right) ,...,m\left( s_{k}\right)
-1 \\
0,\,\,\,\mbox{otherwise}%
\end{array}
\right.
\end{equation*}
where
\begin{equation*}
h_{k}=\left( \frac{1}{2^{k}\sum\limits_{j=1}^{m\left( s_{k}\right) }\left(
1/\lambda _{j}\right) }\right) ^{1/2}.
\end{equation*}

Let
\begin{equation*}
f\left( x,y\right) =\sum\limits_{k=2}^{\infty }f_{k}\left( x\right)
f_{k}\left( y\right) ,\,\,\,\,\,
\end{equation*}%
where
\begin{equation*}
\,f\left( x+l,y+s\right) =f\left( x,y\right) ,\,\,\,\,\,\,\,\,l,s=0,\pm
1,\pm 2,....
\end{equation*}

First we prove that $f\in \Lambda ^{\#}BV.$ For every choice of
nonoverlapping intervals $\{I_{n}:n\geq 1\}$, we get
\begin{eqnarray*}
&&\Lambda ^{\#}V_{1}\left( f;p\left( n\right) \uparrow \infty \right) \leq
\sum\limits_{j=1}^{\infty }\frac{\left\vert f\left( I_{j},y_{j}\right)
\right\vert }{\lambda _{j}} \\
&\leq &4\sum\limits_{i=1}^{\infty }h_{i}^{2}\sum\limits_{j=1}^{m(s_{i})}%
\frac{1}{\lambda _{j}}=4\sum\limits_{i=1}^{\infty }\frac{1}{2^{i}}=4.
\end{eqnarray*}

Analogously, we can prove that
\begin{equation*}
\Lambda ^{\#}V_{2}\left( f;p\left( n\right) \uparrow \infty \right) \leq 4.
\end{equation*}

Next, we shall prove that $f\notin BV^{\#}\left( p\left( n\right) \uparrow
\infty \right) .$ By (\ref{6}), (\ref{7}) and from the construction of the
function we get
\begin{eqnarray*}
&&V_{1}\left( f;p\left( n\right) \uparrow \infty \right) \\
&\geq &\left\{ \sum\limits_{j=m\left( s_{k-1}\right) }^{m\left( s_{k}\right)
-1}\left\vert f\left( \frac{2j-1}{2^{s_{k}}},\frac{2j}{2^{s_{k}}}\right)
-f\left( \frac{2j}{2^{s_{k}}},\frac{2j}{2^{s_{k}}}\right) \right\vert
^{p\left( s_{k}\right) }\right\} ^{1/p\left( s_{k}\right) } \\
&=&\left\{ \sum\limits_{j=m\left( s_{k-1}\right) }^{m\left( s_{k}\right)
-1}\left\vert \left( f_{k}\left( \frac{2j-1}{2^{s_{k}}}\right) -f_{k}\left(
\frac{2j}{2^{s_{k}}}\right) \right) f_{k}\left( \frac{2j}{2^{s_{k}}}\right)
\right\vert ^{p\left( s_{k}\right) }\right\} ^{1/p\left( s_{k}\right) } \\
&=&h_{k}^{2}\left( m\left( s_{k}\right) -m\left( s_{k-1}\right) \right)
^{1/p\left( s_{k}\right) } \\
&\geq &c\frac{m\left( s_{k}\right) ^{1/p\left( s_{k}\right) }}{%
2^{k}\sum\limits_{j=1}^{m\left( s_{k}\right) }\left( 1/\lambda _{j}\right) }%
\geq c2^{k}\rightarrow \infty \,\,\,\,\,\,\,\,\,\,%
\mbox{as\thinspace
\thinspace \thinspace \thinspace \thinspace \thinspace }k\rightarrow \infty .
\end{eqnarray*}%
Therefore, we get $f\notin BV^{\#}\left( p\left( n\right) \uparrow \infty
\right) .$

b)\thinspace \thinspace \thinspace Let
\begin{equation*}
2^{n_{k}-n_{k-1}-1}<m\left( n_{k}\right) \leq 2^{n_{k}}\,\,\,\,\,\,\,\,%
\mbox{for}\,\,\mbox{all\thinspace }\,\,\,\,\,k>k_{0}.
\end{equation*}
Consider the function $g_{k}$ defined by
\begin{equation*}
g_{k}\left( x\right) =\left\{
\begin{array}{l}
d_{k}\left( 2^{n_{k}}x-2j+1\right) ,x\in \left[ \left( 2j-1\right)
/2^{n_{k}},2j/2^{n_{k}}\right) \\
-d_{k}\left( 2^{n_{k}}x-2j-1\right) ,x\in \left[ 2j/2^{n_{k}},\left(
2j+1\right) /2^{n_{k}}\right) \\
\,\,\,\,\,\,\,\,\,\,\,\,\,\,\,\,\,\,\,\,\,\,\,\,\,\,\,\,\,\,\,\,\,\,\,\,\,\,%
\,\,\,\,\,\,\,\,\,\,\,\,\mbox{for\thinspace }%
j=2^{n_{k-1}-n_{k-2}},...,2^{n_{k}-n_{k-1}-1}-1 \\
0,\,\,\,\mbox{otherwise}%
\end{array}
\right.
\end{equation*}
where
\begin{equation*}
d_{k}=\left( \frac{1}{2^{k}\sum\limits_{j=1}^{m\left( n_{k}\right) }\left(
1/\lambda _{j}\right) }\right) ^{1/2}.
\end{equation*}

Let
\begin{equation*}
g\left( x,y\right) =\sum\limits_{k=k_{0}+2}^{\infty }g_{k}\left( x\right)
g_{k}\left( y\right),
\end{equation*}%
where
\begin{equation*}
g\left( x+l,y+s\right) =g\left( x,y\right) ,\,\,\,\,\,\,\,\,\,l,s=0,\pm
1,\pm 2,....
\end{equation*}

For every choice of nonoverlapping intervals $\{I_{n}:n\geq 1\}$ we get
\begin{eqnarray*}
&&\sum\limits_{j=1}^{\infty }\frac{\left\vert f\left( I_{j},y_{j}\right)
\right\vert }{\lambda _{j}} \\
&\leq &4\sum\limits_{i=k_{0}+1}^{\infty
}d_{i}^{2}\sum\limits_{j=1}^{2^{n_{i}-n_{i-1}-1}}\frac{1}{\lambda _{j}} \\
&\leq &4\sum\limits_{i=k_{0}+1}^{\infty }d_{i}^{2}\sum\limits_{j=1}^{m\left(
n_{i}\right) }\frac{1}{\lambda _{j}}<\infty .
\end{eqnarray*}%
Analogously, we can prove that
\begin{equation*}
\sum\limits_{j=1}^{\infty }\frac{\left\vert f\left( x_{j},J_{j}\right)
\right\vert }{\lambda _{j}}<\infty .
\end{equation*}

Hence we have $g\in \Lambda ^{\#}BV.$

Next we shall prove that $g\notin BV^{\#}\left( p\left( n\right) \uparrow
\infty \right) .$ By (\ref{3}), (\ref{5}), (\ref{6}) and from the
construction of the function we get
\begin{eqnarray*}
&&V_{1}^{\#}\left( g;p\left( n\right) \uparrow \infty \right) \\
&\geq &\left\{
\sum\limits_{j=2^{n_{k-1}-n_{k-2}}}^{2^{n_{k}-n_{k-1}-1}-1}\left\vert
g\left( \frac{2j-1}{2^{n_{k}}},\frac{2j}{2^{n_{k}}}\right) -g\left( \frac{2j%
}{2^{n_{k}}},\frac{2j}{2^{n_{k}}}\right) \right\vert ^{p\left( n_{k}\right)
}\right\} ^{1/p\left( n_{k}\right) } \\
&=&\left\{
\sum\limits_{j=2^{n_{k-1}-n_{k-2}}}^{2^{n_{k}-n_{k-1}-1}-1}\left\vert \left(
g_{k}\left( \frac{2j-1}{2^{n_{k}}}\right) -g_{k}\left( \frac{2j}{2^{n_{k}}}%
\right) \right) g_{k}\left( \frac{2j}{2^{n_{k}}}\right) \right\vert
^{p\left( n_{k}\right) }\right\} ^{1/p\left( n_{k}\right) } \\
&=&d_{k}^{2}\left( 2^{n_{k}-n_{k-1}-1}-2^{n_{k-1}-n_{k-2}}\right)
^{1/p\left( n_{k}\right) } \\
&\geq &\frac{1}{4}d_{k}^{2}2^{\left( n_{k}-n_{k-1}\right) /p\left(
n_{k}\right) } \\
&\geq &\frac{c2^{n_{k}/p\left( n_{k}\right) }}{2^{k+2}\sum\limits_{j=1}^{m%
\left( n_{k}\right) }\left( 1/\lambda _{j}\right) } \\
&\geq &c\frac{m\left( n_{k}\right) ^{1/p\left( n_{k}\right) }}{%
2^{k}\sum\limits_{j=1}^{m\left( n_{k}\right) }\left( 1/\lambda _{j}\right) }%
\,\,\,\, \\
&\geq &c2^{k}\,\rightarrow \infty \,\,\,\,\,%
\mbox{as\thinspace \thinspace
\thinspace \thinspace \thinspace \thinspace }k\rightarrow \infty .
\end{eqnarray*}

Therefore, we get $g\notin BV^{\#}\left( p\left( n\right) \uparrow \infty
\right) $ and the proof of Theorem 1 is complete.
\end{proof}

\begin{proof}[Proof of Theorem \protect\ref{th2}]
We choose a monotone increasing sequence of positive integers $\{l_{k}:k\geq
1\}$ such that $l_{1}=1$ and
\begin{equation}
p\left( l_{k-1}\right) \geq \ln k\,\,\,\,\,\,\,\,%
\mbox{for all\thinspace
\thinspace \thinspace \thinspace \thinspace \thinspace \thinspace }k\geq 2.
\label{8}
\end{equation}

Set $\left( k=1,2,...\right) $
\begin{equation*}
r_{k}\left( x\right) =\left\{
\begin{array}{l}
2^{l_{k}+1}c_{k}\left( x-1/2^{l_{k}}\right) ,\,\,%
\mbox{if\thinspace
\thinspace \thinspace }1/2^{l_{k}}\leq x\leq 3/2^{l_{k}+1} \\
-2^{l_{k}+1}c_{k}\left( x-1/2^{l_{k}-1}\right) ,\,\,%
\mbox{if\thinspace
\thinspace \thinspace }3/2^{l_{k}+1}\leq x\leq 1/2^{l_{k}-1} \\
\,\,0,\,\,\,\mbox{otherwise}\,\,\,\,\,\,\,\,\,\,\,\,\,\,\,%
\end{array}%
\right.
\end{equation*}%
where
\begin{equation*}
c_{k}=\left( \sum\limits_{j=1}^{k}\frac{1}{\lambda _{j}}\right) ^{-1/4}
\end{equation*}%
and
\begin{equation*}
r\left( x,y\right) =\sum\limits_{k=1}^{\infty }r_{k}\left( x\right)
r_{k}\left( y\right),
\end{equation*}%
where
\begin{equation*}
r\left( x+l,y+s\right) =r\left( x,y\right) \text{ \qquad }l,s=0,\pm 1,\pm
2,....
\end{equation*}

It is easy to show that function $r\in C\left( I^{2}\right) $.

First we show that $r\in BV^{\#}\left( p\left( n\right) \uparrow \infty
\right) .$ Let $\left\{ I_{i}\right\} $ be an arbitrary partition of the
interval $I$ such that $\inf\limits_{i}\left\vert I_{i}\right\vert \geq
1/2^{l}.$ For this fixed $l$, we can choose integers $l_{k-1}$ and $l_{k}$
for which $l_{k-1}\leq l<l_{k}$ holds. Then it follows that $p\left(
l_{k-1}\right) \leq p\left( l\right) \leq p\left( l_{k}\right) $ and $%
1/2^{l_{k}}<1/2^{l}\leq 1/2^{l_{k-1}}.$

By (\ref{8}) and from the construction of the function $r$ we obtain
\begin{eqnarray*}
&&\left\{ \sum\limits_{j=1}^{m}\left\vert r\left( I_{i},y_{i}\right)
\right\vert ^{p\left( l\right) }\right\} ^{1/p\left( l\right) } \\
&=&\left\{ \sum\limits_{j=1}^{k}\sum\limits_{\left\{ i:2^{-l_{j}}\leq
y_{i}<2^{-l_{j}+1}\right\} }\left\vert r\left( I_{i},y_{i}\right)
\right\vert ^{p\left( l\right) }\right\} ^{1/p\left( l\right) } \\
&\leq &\left\{ \sum\limits_{j=1}^{k}\left( \sum\limits_{\underset{I_{i}\cap
\left( 2^{-l_{j}},2^{-l_{j}+1}\right) \neq \varnothing }{\left\{
i:2^{-l_{j}}\leq y_{i}<2^{-l_{j}+1}\right\} }}\left\vert r\left(
I_{i},y_{i}\right) \right\vert \right) ^{p\left( l\right) }\right\}
^{1/p\left( l\right) }
\end{eqnarray*}%
\begin{eqnarray*}
&\leq &\left\{ \sum\limits_{j=1}^{k}\left( \sum\limits_{\left\{ i:I_{i}\cap
\left( 2^{-l_{j}},2^{-l_{j}+1}\right) \neq \varnothing \right\} }\left\vert
r\left( I_{i},\frac{3}{2^{l_{j}+1}}\right) \right\vert \right) ^{p\left(
l\right) }\right\} ^{1/p\left( l\right) } \\
&\leq &\left\{ \sum\limits_{j=1}^{k}\left( \left\vert r\left( \left( \frac{1%
}{2^{l_{j}}},\frac{3}{2^{l_{j}+1}}\right) ,\frac{3}{2^{l_{j}+1}}\right)
\right\vert +\left\vert r\left( \left( \frac{3}{2^{l_{j}+1}},\frac{1}{%
2^{l_{j}-1}}\right) ,\frac{3}{2^{l_{j}+1}}\right) \right\vert \right)
^{p\left( l\right) }\right\} ^{1/p\left( l\right) } \\
&\leq &\left\{ \sum\limits_{j=1}^{k}\left( 2c_{j}^{2}\right) ^{p\left(
l\right) }\right\} ^{1/p\left( l\right) }\leq 2k^{1/p\left( l_{k-1}\right)
}\leq 4k^{1/\ln k}=4e.
\end{eqnarray*}
Therefore $r\in BV^{\#}\left( p\left( n\right) \uparrow \infty \right) $
holds.

Finaly, we prove that $r\notin \Lambda BV^{\#}.$ Since $c_{n}\downarrow 0,$
we get
\begin{eqnarray*}
&&\sum\limits_{j=1}^{k}\frac{\left\vert r\left(
1/2^{l_{j}},3/2^{l_{j}+1}\right) -r\left( 3/2^{l_{j}+1},3/2^{l_{j}+1}\right)
\right\vert }{\lambda _{j}} \\
= &&\sum\limits_{j=1}^{k}\frac{\left\vert \left( r_{j}\left(
1/2^{l_{j}}\right) -r_{j}\left( 3/2^{l_{j}+1}\right) \right) r_{j}\left(
3/2^{l_{j}+1}\right) \right\vert }{\lambda _{j}} \\
&=&\sum\limits_{j=1}^{k}\frac{c_{j}^{2}}{\lambda _{j}}\geq
c_{k}^{2}\sum\limits_{j=1}^{k}\frac{1}{\lambda _{j}} \\
&=&\left( \sum\limits_{j=1}^{k}\frac{1}{\lambda _{j}}\right)
^{1/2}\rightarrow \infty \,\,\,\,\,%
\mbox{as\thinspace \thinspace \thinspace \thinspace \thinspace
\thinspace }k\rightarrow \infty .
\end{eqnarray*}
%{In the  first estimate we use the following well known estimate $\sum a_i^{p(l)}\leq \left(\sum a_i\right)^{p(l)}$ for $p(l)\geq 1$).
Therefore, we get $r\notin \Lambda BV^{\#}$ and the proof of Theorem \ref{th2}
is complete.
\end{proof}

Since $\Lambda BV^{\#}=B\left( I^{2}\right) $ if and only if $%
\sum\limits_{j=1}^{\infty }\left( 1/\lambda _{j}\right) <\infty $ the
validity of Corollary \ref{cor} follows from Theorem \ref{th2}.

{\bf Acknowledge }. We thanks to the anonymous referee for his/her remarks which have improved the final version of this paper.

\end{document}